\documentclass{amsart}
\usepackage{amsthm,amssymb,latexsym,amsmath}

\def \C {\mathbb{C}}

\def \DC {\widetilde{\Delta}}

\def \ddd {\mathcal{D}}
\def \dps {\displaystyle}

\def \fol {{\mathcal F}}
\def \F {{\mathcal F}}

\def \ll {\mathcal{L}}

\def \nn {\mathcal{N}}

\def \oo {\mathcal{O}}

\def \P {\mathbb{P}}
\def \pn {\mathbb{P}^n}
\def \pnt {\widetilde{\mathbb{P}}^n}

\def \sing {{\rm Sing}}
\def \singf {{\rm Sing}(\mathcal{F})}

\def \tt {\mathcal{T}}

\newtheorem{proposition}{Proposition}[section]

\newtheorem{corolary}[proposition]{Corollary}
\newtheorem{lema}[proposition]{Lemma}
\newtheorem{theorem}[proposition]{Theorem}
\newtheorem{example}[proposition]{Example}
\begin{document}

\title{Foliations by curves with curves as singularities}

\author{M.  Corr\^ea Jr. }
\address{\noindent Maur\' \i cio Corr\^ea Jr: Universidade Federal de Vi\c cosa,
Departamento de Matem\'atica, Avenida P.H. Rolfs, s/n, 36570--000,
Vi\c cosa-MG, Brazil} \email{mauricio.correa@ufv.br}

\author{A. Fern\'andez-P\'erez}
\address{Arturo Fernandez Perez \\
ICEx - UFMG \\
Departamento de Matem\'atica \\
Av. Ant\^onio Carlos 6627 \\
30123-970 Belo Horizonte MG, Brazil} \email{arturofp@mat.ufmg.br}

\author{G. Nonato Costa}
\address{Gilcione Nonato Costa \\
ICEx - UFMG \\
Departamento de Matem\'atica \\
Av. Ant\^onio Carlos 6627 \\
30123-970 Belo Horizonte MG, Brazil} \email{gilcione@mat.ufmg.br}

\author{R. Vidal Martins}
\address{Renato Vidal Martins \\
ICEx - UFMG \\
Departamento de Matem\'atica \\
Av. Ant\^onio Carlos 6627 \\
30123-970 Belo Horizonte MG, Brazil}
\email{renato@mat.ufmg.br}


\maketitle

\dedicatory{\it\ \ \ \ \ \ \ \ \ \ \ \ \ Dedicated to M\'arcio Gomes Soares, for his 60th birthday}

\begin{abstract}
Let $\fol$ be a holomorphic one-dimensional foliation on $\pn$ such that the components of its singular locus $\Sigma$ are curves $C_i$ and points $p_j$. We determine the number of $p_j$, counted with multiplicities, in terms of invariants of $\fol$ and $C_i$, assuming that $\F$ is special along the $C_i$. Allowing just one nonzero dimensional component on $\Sigma$, we also prove results on when the foliation happens to be determined by its singular locus.
\end{abstract}


\section{Introduction}

Let $\fol$ be a foliation on a smooth projective scheme $Y$, and $X$ a projective subscheme of $Y$. Let $\widetilde{Y}$ be the blowup of $Y$ along $X$, and $\pi:\widetilde{Y}\rightarrow Y$ the blowup morphism with
exceptional divisor $E:=\pi^{-1}(X)$. The foliation $\fol$
will be called {\it special } along $X$ if the strict transform $\widetilde{\fol}$ has
$E$ as an invariant set, and ${\rm Sing}(\widetilde{\fol})$ meets $E$ at isolated singularities at most. With this in mind, we prove the following.


\

\noindent {\bf Theorem 1.} \emph{Let $\fol$ be a holomorphic foliation by curves on
$\pn$, $n\geq 3$, of degree $k$, such that its singular locus is the disjoint union of irreducible curves $C_1,\ldots,C_r$ and points $p_1,\ldots,p_s$. Assume each $C_i$ is either smooth, or a singular set theoretic complete intersection; assume also that $\fol$ is special along each $C_i$, for $1\leq i\leq r$. Then
$$
\sum_{i=1}^{s}\mu(\fol,p_i) = 1+k+k^2+\ldots+k^n+\sum_{i=1}^r \nu(\fol,C_i)
$$
where $\mu(\fol,p_i)$ is the multiplicity of $\fol$ at $p_i$, and where for any curve $C\subset\pn$ of arithmetic genus $g$, degree $d$, with singular points, if any, $q_1,\ldots,q_l$, and along which $\fol$ is special, we set
$$
\nu(\fol,C):=(\ell+1)^{n-2}\bigg(\bigg(2g-2-\sum_{i=1}^l (b_i-1)\bigg)(\ell^2+\ell+1)+(n+1)d\ell^2-(k-1)d(n\ell+1)\bigg)
$$
with $b_i$ the number of branches of $q_i$, and $\ell:=m_C(\fol)$ the multiplicity of $\fol$ at $C$.}

\

The above result generalizes a formula by the third named author (cf. \cite{GNC}) which counts the number of isolated singularities of a foliation by curves on $\P^3$ admiting regular curves as singularities. We just relaxed the hypothesys basically allowing singular curves on the singular locus, as long as complete intersections, and consider the foliation on a generic projective space $\P^n$.

Our second task concerns stablishing conditions for when the singular locus happens to determine a foliation. In order to do so, we need another definition. Let $X$ be a projective subscheme of $\pn$ given by an ideal sheaf $\mathcal{I}_X$. We define the \emph{generating degree of $X$}, denoted ${\rm gd}(X)$, as the least integer $d>0$ such that $\mathcal{I}_X(d)$ is globally
generated, i.e., for which $X$ is a set theoretic intersection of hypersurfaces of degree at most $d$.

\

\noindent {\bf Theorem 2.} \emph{Let $\F$ be a holomorphic foliation
by curves on $\pn$, $n\geq 3$, of degree $k$, such that its singular locus has just one nonzero dimensional component, which is an integral and smooth curve $C$. Assume also that $\F$ is special along $C$. Let $\pi: \widetilde{\mathbb{P}}^n\rightarrow\pn$ be the blowup of $\pn$ along $C$ and $E$ the exceptional divisor. If $\F'$ is another foliation of degree $k$ on $\pn$, with  $k>{\rm gd}(C)$, and also
${\rm Sing}(\F)\subset{\rm Sing}(\F')$ and ${\rm Sing}(\widetilde{\F}|_{E})\subset{\rm Sing}(\widetilde{\F}'|_{E})$, then $\F'=\F$.}

\

The above result can be compared to A. Campillo and J. Olivares' \cite[Cor. 3.2]{C-O}, the proof of which, along with X. Gomez-Mont and G. Kempf's \cite{GK}, motivated the one here. In the very case of three dimensional ambient space, with additional requirements on the curve of singularities, a stronger sentence can be proved.

\

\noindent {\bf Theorem 3.}
\emph{If $\F$ and $\F'$ are holomorphic foliations by curves on $\mathbb{P}^3$, of same degree, such that
${\rm Sing}(\F')\supset{\rm Sing}(\F)=C\cup\{p_1,\ldots,p_s\}$, where $C$ is a nondegenerated integral smooth set theoretic complete intersection curve, and $\F$ is special along $C$, then $\F=\F'$.}

\

So in this case we have the desired statement on determination. This unicity problem has been studied by many authors (see \cite{AC} and the references therein) and, in this sense, the above result is a step forward on dealing with the subject when the singular locus has unexpected codimension.

\

\paragraph{\bf Acknowledgments.}
The second named author is partially supported by Fapemig-Brasil and PRPq-UFMG; the fourth named author is partially supported by the CNPq grant number 304919/2009-8.

\section{Preliminaries}

\subsection{Multiplicities along subschemes}



Let $\fol$ be a foliation by curves on an
$n$-dimensional smooth projective variety $Y$ over $\C$. It is determined by an injective sheaf map $\varphi:\mathcal{L}_{\F}\hookrightarrow\mathcal{T}_Y$, where $\mathcal{L}_{\F}$ is an invertible sheaf and $\mathcal{T}_Y$ is the tangent bundle, such that $\mathcal{T}_Y/\mathcal{L}_{\F}$ is torsion free. The \emph{singular locus of $\fol$} is the closed subscheme $\Sigma$ of $Y$ defined by the ideal sheaf
$$
\mathcal{I}_{\Sigma}:=F_{n-1}(\mathcal{T}_Y/\mathcal{L}_{\F})
$$
where $F_{n-1}$ stands for the fitting ideal. So one may denote and write
$$
\singf:=\Sigma={\rm Spec}\,\mathcal{O}_Y/\mathcal{I}_{\Sigma}.
$$

Now let $P\in Y$ be any point. The local ring $\oo_{Y,P}$ may not be a discrete valuation ring, but one can at least consider the $\mathfrak{m}_{P}$-adic valuation, which we denote by $v_P$, where $\mathfrak{m}_{P}$ is the maximal ideal. So one defines the \emph{multiplicity of $\F$ at $P$} as
$$
m_{P}(\F):=\min\limits_{f\in\mathcal{I}_{\Sigma,P}}\{v_P(f)\}
$$
The multiplicity of $\F$ at an irreducible subscheme of $Y$ will be the multiplicity of the foliation at its generic point. So multiplicities are well defined for irreducible components of the singular locus as well.

If $\fol$ is given by a vector field which, in a neighbourhood of $P$, is written by
\begin{equation}
\label{equcmp}
\ddd_{\F}=\ddd_{\F,P}=f_{1}\,\frac{\partial}{\partial z_{1}}+\ldots+f_{n}\,\frac{\partial}{\partial z_{n}}
\end{equation}
then we have
$$
m_{P}(\F):=\min\{v_P(f_1),\ldots,v_P(f_n)\}.
$$
If $p\in Y$ is a closed point then
$$
m_p(\fol) = {\rm length}_{\,\oo_{Y,p}}\,\frac{{\mathcal O}_{Y,p}}{(f_1,\ldots,f_n)}
$$
which agrees with the classical Milnor number. So in this case we adopt the standard notation
$$
\mu(\fol,p)=m_p(\fol).
$$

For later use, we now describe the multiplicity of $\F$ at an irreducible curve $C$ which is a component of $\singf$. By a holomorphic change of coordinates, $C$ can be locally given as
$z_1=\ldots=z_{n-1}=0$. Therefore, one may write the local sections in (\ref{equcmp}) as
\begin{equation}
\label{for40}
        f_i(z) = \dps\sum_{|a|=m_i}z_1^{a_1}\cdots z_{n-1}^{a_{n-1}}f_{i,a}(z)
\end{equation}
where $a:=(a_1,\ldots,a_{n-1})$ with $|a|:=a_1+\ldots+a_{n-1}$, and at least one among the $f_{i,a}(z)$ does not vanish in the $z_n$-axis.
One rapidly sees that the number $m_i$ in (\ref{for40}) agrees with $v_C(f_i)$ so
\begin{equation}
\label{eqummm}
m_{C}(\fol)={\rm min}\{m_1,\ldots,m_n\}.
\end{equation}
We may change coordinates and assume for the remainder that
$$
m_{n-1}\leq \ldots\leq m_1.
$$

Now we blowup $Y$ along $C$ and describe the behavior of $\fol$ under this transformation. Just in order to fix notation we recall the blowup procedure in this specific case. If $\Delta$ is an $n$-dimensional
polydisc with holomorphic coordinates $z_1,\ldots,z_n$ and $\Gamma\subset
\Delta$ is the locus $z_1=\ldots=z_{n-1}=0$, take $[y_1,\ldots,y_{n-1}]$ to be
homogeneous coordinates on $\P^{n-2}$. The blowup of $\Delta$ along $\Gamma$ is the smooth variety
$$
\DC = \lbrace (z,[y])\in \Delta \times \P^{n-2}\, |\, z_i y_j=z_jy_i\ {\rm for}\ 1\le i,j \le n-1\rbrace.
$$
The projection $\pi : \DC \rightarrow \Delta$ on the first factor is an isomorphism away from $\Gamma$, while the inverse image of a point $z\in \Gamma$ is a projective space $\P^{n-2}.$ The inverse image $E=\pi^{-1}(\Gamma)$ is the exceptional divisor of the blowup.

The standard open cover $U_j =
\lbrace [y_1,\ldots,y_{n-1}] \,|\, y_j\ne 0\rbrace$, with $1\leq j\leq n-1$, of $\P^{n-2}$ yields a cover of $\widetilde{\Delta}$ where each open set, for $1\leq j\leq n-1$, is defined by
\begin{equation}
\label{for100}
{\widetilde{U_j}}= \lbrace(z,[y]) \in \DC\ \, |\ \,
[y]\in U_j\rbrace
\end{equation}
with holomorphic coordinates $\sigma(u_1,\ldots,u_n)=(z_1,\ldots,z_n)$ given by
$$z_i=
\begin{cases}
u_i & {\rm if}\ i=j\ \mbox{or}\ i=n \\
u_i u_j & {\rm if}\ i=1,\ldots,{\widehat{j}},\ldots,n-1.
\end{cases}
$$
The coordinates $u\in \C^n$ are affine coordinates on each fiber $\pi^{-1}(p)\cong\P^{n-2}$ of $E$.

Now consider the curve $C\subset Y$. Let $\lbrace
\phi_{\lambda},U_{\lambda}\rbrace$ be a collection of local charts
covering $C$ and $\phi_{\lambda}: U_{\lambda}\to \Delta_{\lambda},$
where $\Delta_{\lambda}$ is an $n$-dimensional polydisc. One may
suppose that $\Gamma_{\lambda}=\phi_{\lambda}(C\cap U_{\lambda})$ is given
by $z_1=\ldots=z_{n-1}=0.$ Let $\pi_{\lambda}: \DC_{\lambda}\to
\Delta_{\lambda}$ be the blowup of $\Delta_{\lambda}$ along
$\Gamma_{\lambda}$. One can patch the $\pi_{\lambda}$ together and use the chart maps $\phi_{\lambda}$ to get a blowup $\widetilde{Y}$ of $Y$ along $C$ and a blowup morphism $\pi : \widetilde{Y}\to Y$. The
exceptional divisor $E$ is a fibre bundle over $C$
with fiber $\P^{n-2}$ which is naturally identified with the
projectivization $\P(\nn_{C/Y})$ of the normal bundle $\nn_{C/Y}$.

In the open set ${\widetilde U}_1$, as in
(\ref{for100}), we have
$$
\sigma(u) = (u_1,u_1u_2,\ldots,u_1u_{n-1},u_n) = (z_1,\ldots,z_n)
$$

If $i=1$ or $i=n$, since $u_i=z_i$ we get
\begin{align*}
\dot u_i &= \dps \sum_{|a|=m_i} u_1^{a_1}(u_1u_2)^{a_2}\cdots(u_1u_{n-1})^{a_{n-1}}f_{i,a}(\sigma(u)) \\
 & = u_1^{m_i}\dps\sum_{|a|=m_i}u_2^{a_2}\cdots u_{n-1}^{a_{n-1}} f_{i,a}(\sigma(u))
\end{align*}
but we may write $f_{i,a}(\sigma(u))=f_{i,a}(0,\ldots,0,u_n) +
u_1{\widetilde f}_{i,a}(u)= p_{i,a}(u_n)+ u_1
\widetilde{f}_{i,a}(u)$ and hence
\begin{equation}
\label{equ1en}
\dot u_i = u_1^{m_i} \left( \dps
\sum_{|a|=m_i}u_2^{a_2}\cdots u_{n-1}^{a_{n-1}} p_{i,a}(u_n)
+ u_1 \widetilde{f}_i(u)\right)
\end{equation}
for some functions $\widetilde{f}_i(u)$, with $i=1$ or $i=n$.

If $2\leq i\leq n-1$, since $z_i=u_1u_i$, we have that $\dot z_i=\dot u_1
u_i +u_1\dot u_i$ and thus
\begin{align}
\label{equnm2}
\dot u_i&=
u_1^{m_i-1}\bigg(\dps\sum_{|a|=m_i}u_2^{a_2}\ldots u_{n-1}^{a_{n-1}}
p_{i,a}(u_n)
\\ &\ \ \ \ \ \ \ \ \ \ \ \
-u_1^{m_1-m_i}u_i\dps\sum_{|a|=m_1}u_2^{a_2}\ldots u_{n-1}^{a_{n-1}}
p_{1,a}(u_n)+ u_1\widetilde{f}_i(u)\bigg) \nonumber
\end{align}
for some functions $\widetilde{f}_i(u)$, with $2\leq i\leq n-1$.

Combining (\ref{equ1en}) and (\ref{equnm2}) we have that $\pi^*\F$ is described by the vector field
\begin{align}
\label{for4}
\ddd_{\pi^*\F} &= u_1^{m_1}\bigg( g_1(u)+u_1\widetilde{f}_1(u)\bigg)\frac{\partial}{\partial
  u_1} +\sum_{i=2}^{n-1}u_1^{m_i-1}\bigg( h_i(u)+u_1\widetilde{f}_i(u)\bigg)\frac{\partial}{\partial u_i}\\
  &\ \ \ + u_1^{m_n}\bigg( g_n(u)
  +u_1\widetilde{f}_n(u)\bigg)\frac{\partial}{\partial
  u_n} \nonumber
\end{align}
where
$$
g_i(u):=\dps
\sum_{|a|=m_i}u_2^{a_2}\cdots u_{n-1}^{a_{n-1}}
p_{i,a}(u_n)\ \ \ \ {\rm and}\ \ \ \ h_i(u):=
g_i(u)-u_1^{m_1-m_i}u_ig_1(u)
$$
Now all points of $E$, given by $u_1=0$, are singularities of
$f^{*}\fol$. We have some ways of desingularizing, according to
the possible values of $m_i$. Furthermore, if $m_1=m_{i}$ for some
$i$, we must verify whether
\begin{equation}
\label{form3}\dps
r_{i}(u):=g_i(u) -u_i g_1(u)
\end{equation}
is identically zero or not. In this way, we may divide it in two cases,
dicritical or nondicrital curves of singularities, according to
if the exceptional divisor is or is not invariant by the
induced foliation $\widetilde{\fol}$.
\begin{itemize}
 \item Non-dicritical curve of singularities.
\end{itemize}
\begin{itemize}
\item[(i)] $m_n + 1 = m_{n-1}=\ldots=m_2$ with $r_i\not\equiv
0$ for all $2\leq i\leq n-1$ if $m_{n-1}=m_1$.
\end{itemize}
Dividing (\ref{for4})  by $u_1^{m_n}$ we
get the vector field defining $\widetilde{\fol}$ which is
\begin{equation}
\label{for5}
\ddd_{\widetilde{\fol}}= u_1^{m_1-m_n}(g_1(u)+
u_1 \widetilde{f_1}(u))\frac{\partial}{\partial
  u_1}+\dps\sum_{i=2}^{n-1}\widetilde{h}_i(u)\frac{\partial}{\partial
  u_i}+ ( g_n(u)+u_1\widetilde{f}_n(u))\frac{\partial}{\partial
  u_n}
\end{equation}
where
$$
\widetilde{h}_i(u):=g_i(u)-u_1^{m_1-m_i}u_ig_1(u)+u_1\widetilde{f}_i(u).
$$
The singularities on $E$ are given by the roots of the system
$$
            \widetilde{h}_2(u)=
            \widetilde{h}_3(u)=\cdots=\widetilde{h}_{n-1}(u)=g_n(u)=0
$$
and (i) implies that they should be isolated, i.e., $\fol$ is special along $C$.
\begin{itemize}
\item[(ii)] $m_n + 1  \leq m_{n-1}$ with $r_{i_0}\equiv 0$ for some $2\leq i_0\leq n-1$ if $m_n + 1 = m_1$.
\end{itemize}
\noindent Dividing (\ref{for4}) by $u_1^{m_n}$ we get
\begin{align}
\label{fori2}
\ddd_{\widetilde{\fol}}&= u_1^{m_1-m_n}\big( g_1(u)+ u_1\widetilde{f_1}(u)\big)\frac{\partial}{\partial u_1}
+\dps\sum_{i=2}^{n-1}u_1^{m_i-m_n-1}\big( h_i(u)+u_1\widetilde{f}_i(u)\big) \frac{\partial}{\partial
u_i} \\
&\ \ \ \ + \big( g_n(u)+u_1\widetilde{f}_n(u)\big)\frac{\partial}{\partial
u_n} \nonumber
\end{align}
and we see that the exceptional divisor is also invariant by the foliation $\widetilde{\fol}$, but
this turn, the singularities are always nonisolated.
Furthermore, the leaves of (\ref{fori2}) when restricted to
$E$ are contained in the hyperplane given by $u_i =c_i$ for those $i$ such that $m_i-1>m_n$ or $r_i\equiv 0$, where $c_i$ is a
constant.
\begin{itemize}
\item[(iii)] $m_{n}\geq m_{n-1}$ with $r_{i_0}\not\equiv0$ for some $2\leq i_0\leq n-1$ if $m_n=m_{n-1}=m_1$.
\end{itemize}
\noindent Dividing (\ref{for4}) by $u_1^{m_{n-1}-1}$ we get
\begin{align}
\label{forii2}
\ddd_{\widetilde{\fol}} &=
u_1^{m_1-m_{n-1}+1}\bigg(g_1(u)+u_1\widetilde{f}_1(u)\bigg)\frac{\partial}{\partial
u_1}+\dps\sum_{i=2}^{n-1}u_1^{m_i-m_{n-1}}\widetilde{h}_i(u)\frac{\partial}{\partial
u_i}\\
&\ \ \ \ +u_1^{m_n-m_{n-1}+1}\bigg(g_n(u)+u_1\widetilde{f}_n(u)\bigg)\frac{\partial}{\partial
u_n} \nonumber
\end{align}
and the exceptional divisor is also invariant by the foliation
$\widetilde{\fol}$, but again with nonisolated singularities on
$\widetilde{C}$. The leaves of $\widetilde{\fol}$ on $E$
are contained in the hyperplane $u_n=c$ for a constant $c$.
\begin{itemize}
\item Dicritical curve of singularities.
\end{itemize}
\begin{itemize}
\item[(i)] $m_1=\cdots=m_n$ and $r_{i}\equiv 0$ for all $2\leq i\leq n-1$.
\end{itemize}
\noindent Dividing (\ref{for4}) by $u_1^{m_n}$ we get
\begin{equation}
\label{forbii}
\ddd_{\widetilde{\fol}}=
\bigg(g_1(u)+u_1\widetilde{f}_1(u)\bigg)\frac{\partial}{\partial
u_1}+\dps\sum_{i=2}^{n-1}\widetilde{f}_i(u)\frac{\partial}{\partial
u_i}+\bigg(g_n(u)+u_1\widetilde{f}_n(u)\bigg)\frac{\partial}{\partial
u_n}
\end{equation}
Combining this with the corresponding expression in the other
coordinate systems, we get defining  equations for a foliation
$\widetilde{\fol}$ which coincides with $f^*\fol$ outside
$E$ but this time the exceptional divisor is not an
invariant set. The foliation ${\widetilde \fol}$ is transverse to
$E$ except at the hypersurface locally given by
$g_1(u) =0$ which may or may not consist of singularities of
$\widetilde{\fol}.$
\begin{itemize}
\item[(ii)] $m_{n-1}=\ldots=m_1 < m_n$ and $r_{i}\equiv 0$ for all $2\leq i\leq n-1$.
\end{itemize}
\noindent Dividing (\ref{for4}) by $u_1^{m_1}$ we get
\begin{equation}
\label{forbiii}
 \ddd_{\widetilde{\fol}}=
\bigg( g_1(u)+u_1\widetilde{f}_1(u)\bigg)\frac{\partial}{\partial
u_1}+\dps\sum_{i=2}^{n-1}\widetilde{f}_i(u)\frac{\partial}{\partial
u_i}+u_1^{m_n-m_1}\bigg( g_n(u)+u_1\widetilde{f}_n(u)\bigg)\frac{\partial}{\partial
u_n}
\end{equation}
and the exceptional divisor is not invariant by the foliation $\widetilde{\fol}$, but the last component of the
vector field (\ref{forbiii}) vanishes on it.

Keeping the notation above, for later use we sketch what we get as follows.

\begin{lema}
\label{lemlll}
The following hold:
\begin{itemize}
\item[(i)] $m_{C}(\fol)={\rm min}\{m_1,\ldots,m_n\}$;
\item[(ii)] if $\ell$ is the integer such that
$$
\ll_{\widetilde{\fol}} \cong \pi^*\ll_{\fol} \otimes \oo_{\widetilde{Y}}(\ell E)
$$
then
\begin{equation*}
\ell=
\begin{cases}
{\rm min}\lbrace m_1, m_2-1,\ldots,m_{n-1}-1, m_n \rbrace & \mbox{ if } C \mbox{ is nondicritical} \\
{\rm min}\lbrace m_1,\ldots,m_n\rbrace & \mbox{ if } C \mbox{ is dicritical }
\end{cases}
\end{equation*}
\item[(iii)] $\fol$ is special along $C$ if and only if $m_n + 1 = m_{n-1}=\ldots=m_2$ with $r_i\not\equiv 0$ for all $2\leq i\leq n-1$ if $m_{n-1}=m_1$. In particular, $\ell=m_{C}(\fol)$ in this case.
\end{itemize}
\end{lema}

\subsection{Chern classes} Now we relate cohomology groups of schemes and blowups. Let $\pi:\pnt\to \pn$, $n \ge 3$, be the blowup of $\P^n$ along a regular curve $C$, with exceptional divisor $E$. Set $\nn:=\nn_{C/\P^n}$ and $\rho:=\pi|_{E}$. Since $E \cong \P(\nn)$, recall that $A(E)$ is generated as an $A(C)$-algebra by the Chern class
$$
\zeta := c_1(\oo_{\nn}(-1))
$$
with the single relation
\begin{equation}
\label{equdch}
\zeta^{n-1} - \rho^*c_1(\nn)\zeta^{n-2}+\ldots+(-1)^{n-1}\rho^*c_{n-2}(\nn)\zeta + (-1)^{n-1}\rho^*c_{n-1}(\nn)=0.
\end{equation}
The normal bundle $\nn_{E/\pnt}$ agrees with the tautological bundle $\oo_{\nn}(-1)$, and hence
\begin{equation}
\label{equzet}
\zeta=c_1(\nn_{E/\pnt}).
\end{equation}
If $\iota :E\hookrightarrow \pnt$ is the inclusion map, we also get
\begin{equation}
\label{equslf}
\iota_{*}(\zeta^i)=(-1)^i E^{i+1}.
\end{equation}
Given that
$$
\dps\int_{E}\rho^*c_{i}(\nn)\zeta^{n-i-1} =
(-1)^{n-i-1}\dps\int_{C}c_i(\nn) = 0
$$
for $i\ge 2$, we have
\begin{align}
\label{forEE}
\int_{E}\zeta^{n-1} &= \int_{E}\rho^*c_1(\nn)\zeta^{n-2} = (-1)^{n} \int_{C}c_1(\nn) \\
                    &=(-1)^n\int_{C}c_1(\tt_{\P^n}\otimes\oo_C)-c_1(C)= \dps(-1)^n\bigg((n+1)d -
2 + 2g\bigg) \nonumber
\end{align}
where $g$ is the genus and $d$ is the degree of $C_{\rm red}$.

From Porteous Theorem (see \cite{IP}), it holds that
\begin{equation}
\label{relcc} c(\pnt) - \pi^*c(\pn) = \iota_*(\rho^*c(C)\alpha)
\end{equation}
where
\begin{equation}
\label{equpor}
\alpha =\frac{1}{\zeta}\dps\sum_{i=0}^{n-1}\big{(}1-(1-\zeta)(1+\zeta)^i\big{)}\rho^*c_{n-1-i}(\nn).
\end{equation}
We may rewrite (\ref{equpor}) taking $(1+\zeta)^i=\sum_{l=0}^{i}{{i}\choose{l}}\zeta^l$ and setting $j:=n-1-i$ as
\begin{equation}
\label{equalf}
\alpha =
\dps\sum_{j=0}^{n-1}\sum_{l=0}^{n-1-j}\bigg({{n-1-j}\choose{l}}-{{n-1-j}\choose{l+1}}\bigg)\zeta^{l}\rho^*c_j(\nn)
\end{equation}
with the convention, also for the remainder, that ${{p}\choose{q}}:= 0$ whenever $q>p$. Since $i$ does not appear in (\ref{equalf}) we reset $i:=j+l$ and write
$$
\alpha=\sum_{i=0}^{n-1}\alpha_i
$$
where
$$
\dps\alpha_i =
\sum_{j=0}^{i}\bigg({{n-1-j}\choose{i-j}}-{{n-1-j}\choose{i-j+1}}\bigg)\zeta^{i-j}
\rho^*c_j(\nn).
$$
Consequently
$$
c(\pnt)-\pi^*c(\pn)=\iota_*(\rho^*c(C)
\alpha)=\iota_*\bigg(\sum_{i=0}^{n-1} \beta_i\bigg)
$$
where
$$
\beta_i=\dps\sum_{j=0}^{i}\alpha_j \rho^*c_{i-j}(C).
$$
Then $\beta_0 = \alpha_0 = -(n-2)$ and
$\beta_i = \alpha_i + \alpha_{i-1} \rho^*c_1(C)$ for $i\ge 1$.
Now,  in order to calculate the Chern class $c(\pnt)$ we
have to compare the terms of (\ref{relcc}) of same degree.
Therefore
$$
c_i(\pnt)-\pi^*c_i(\pn)=\iota_*(\beta_{i-1}).
$$
which yields
\begin{equation}
\label{for20} c_1(\pnt) - \pi^*c_1(\pn) = \iota_*(\beta_0) =
-(n-2)E
\end{equation}
and for $i\ge 2$,
\begin{align}
\label{fcc}
c_i(\pnt) &=\pi^*c_i(\pn)+\dps\sum_{j=0}^{i-1}\bigg({{n-1-j}\choose{i-1-j}}-{{n-1-j}\choose{i-j}}\bigg)(-1)^{i-1-j}\rho^*c_j(\nn)E^{i-j} \\
                   &+\dps\sum_{j=0}^{i-2}\bigg({{n-1-j}\choose{i-2-j}}-{{n-1-j}\choose{i-j-1}}\bigg)(-1)^{i-2-j}
\rho^*c_j(\nn)\rho^*c_1(C)E^{i-1-j}.\nonumber
\end{align}

\section{Special Foliations along Regular Curves}

In this section, $\fol$ is always a holomorphic
foliation by curves on $\P^n$, $n\ge 3$, with
\begin{equation}
\label{for13} \mbox{Sing}(\fol) = \dps C \cup
\{p_1,\ldots,p_s\},
\end{equation}
where the union is disjoint, $C$ is an irreducible smooth projective curve, the $p_i$ are isolated closed points, and $\fol$ is special along $C$. This means that for the blowup $\pi:\pnt\to\P^n$ along $C$, we obtain a foliation $\widetilde{\fol}$ on $\widetilde{\P}^n$ which has only isolated singularities, and the exceptional divisor $E$ is an invariant set of $\widetilde{\fol}$.

Our goal is to compute the number of isolated
singularities of $\fol$, counted with multiplicities. We assume (\ref{for13}) for simplicity, but the general case where $\singf$ has more than one curve as a component is straight forward from this one. The case where $C$ is a singular set theoretic complete intersection is left to the following section.

We start by calculating the Chern class of the invertible sheaf $\ll_{\widetilde{\fol}}$, the tangent bundle of the foliation $\widetilde{\fol}$.
From Lemma \ref{lemlll}, it follows that
$$
\ll_{\widetilde{\fol}} \cong \pi^*\ll_{\fol} \otimes \oo_{\pnt}(\ell E)
$$
where $\ell = m_C(\fol)$. Therefore
\begin{equation}
\label{prop03} c_1(\ll_{\widetilde{\fol}}) = \pi^*c_1(\ll_{\fol}) +
\ell E
\end{equation}
The result below is the first step to get Theorem 1, announced in the Introduction.

\begin{theorem}
\label{teoA}
Let $\fol$ has degree $k$, and  multiplicity $\ell$ at $C$; let $C$ has genus $g$ and degree $d$; and let ${\rm Sing}(\widetilde{\fol}|_{E})=\{\widetilde{q}_1,\ldots,\widetilde{q}_t\}$. Then
\begin{align*}
\sum_{i=1}^t \mu(\widetilde{\F}|_{E},\widetilde{q}_i) &=(2-2g)\bigg(1+(\ell+1)+(\ell+1)^2+\ldots+(\ell+1)^{n-3}\bigg)\\
&\ \ \ +(\ell+1)^{n-2}\bigg((2-2g)(\ell+1)-(n+1)d\ell+(k-1)d(n-1)\bigg)
\end{align*}
\end{theorem}


\begin{proof}
By Baum-Bott's formula \cite{BB}, we have that
$$\sum_{i=1}^t \mu(\widetilde{\F}|_{E},\widetilde{q}_i)= \int_{E}c_{n-1}(\tt_{E} \otimes \ll_{\widetilde{\fol}}^*)$$
with
$$
c_{n-1}(\tt_{E} \otimes \ll_{\widetilde{\fol}}^*) = \sum_{i=0}^{n-1}c_i(E)\cdot c_1(\ll_{\widetilde{\fol}}^*)^{n-i-1}.
$$
On the one hand,
\begin{equation}
\label{equwit}
c_i(E) = c_i(\tt_{\pnt}\otimes\oo_E)-c_{i-1}(E) \zeta
\end{equation}
and reaplying (\ref{equwit}) recursively we obtain
$$
c_i(E)=\dps\sum_{j=0}^{i}(-1)^jc_{i-j}(\tt_{\pnt}\otimes\oo_E)
\zeta^{j}.
$$
Set $\rho:=\pi|_{E}$ and $\nn:=\nn_{C/\pn}$. Then, using also (\ref{fcc}), for $i\ge 1$ we get
\begin{align*}
c_{i}(E)&=\sum_{j=0}^{i-1}(-1)^j\pi^*c_{i-j}(\pn)\zeta^j +(-1)^i
{{n-1}\choose{i}}\zeta^{i} \\
&\ \ \ +\dps\sum_{j=1}^{i-1}(-1)^{i-j-1}\bigg(1-{{n-j-1}\choose{i-j}}\bigg)\rho^*c_{j}(\nn)\zeta^{i-j} \\
&\ \ \ \dps+\sum_{j=0}^{i-2}(-1)^{i-j}\bigg(1-{{n-j-1}\choose{i-j-1}}\bigg)\rho^*c_{j}(\nn)\rho^*c_1(C)\zeta^{i-j-
1}.
\end{align*}
On the other hand, as $c_1(\ll_{\widetilde{\fol}}^*)=\pi^*c_1(\ll_{\fol}^*) -\ell E$, we have
$$c_1(\ll_{\widetilde{\fol}}^*)^{n-i-1}=\dps\sum_{l=0}^{n-i-1}{{n-i-1}\choose{l}}\pi^*c_1(\ll_{{\fol}}^*)^{l}(-\ell E)^{n-i-l-1}.$$
Passing from $E$ to $C$ one rapidly sees that
$$
\dps
\int_{E}\pi^*c_{i-j}(\pn)\pi^*c_1(\ll_{\fol}^*)^l\zeta^{n-i+j-l-1}=0\ \ \ {\rm for}\ j\leq i-2\ {\rm or}\ l\geq 1
$$
$$
\dps
\int_{E}\pi^*c_1(\ll_{\fol}^*)^l\zeta^{n-l-1}=0\ \ \ {\rm for}\ l\geq 2
$$
$$
\dps
\int_{E}\rho^*c_{j}(\nn)\pi^*c_1(\ll_{\fol}^*)^l\zeta^{n-j-l-1}=0\ \ \ {\rm for}\ j\geq 2\ {\rm or}\ l\geq 1
$$
$$
\dps
\int_{E}\rho^*c_{j}(\nn)\rho^*c_1(C)\pi^*c_1(\ll_{\fol}^*)^l\zeta^{n-j-l-2}=0\ \ \ {\rm for}\ j\geq 1\ {\rm or}\ l\geq 1
$$
and we obtain for $i \ge 1$,
\begin{align*}
\dps\int_{E}c_i(E)\cdot
c_1(\ll_{\widetilde{\fol}}^*)^{n-i-1} & = \dps(-1)^n\ell^{n-i-1}\int_{E}\pi^*c_1(\pn)\zeta^{n-2} \\
 &\ \ \ + \dps(-1)^{n-1}\ell^{n-i-1}{{n-1}\choose{i}}\int_{E}\zeta^{n-1}\\
 &\ \ \ + (-1)^n\ell^{n-i-2}{{n-1}\choose{i}}{{n-i-1}\choose{1}}\int_{E}\pi^*c_1(\ll_{\fol}^*)\zeta^{n-2}\\
 &\ \ \ +\dps(-1)^{n-1}\ell^{n-i-1}\bigg(1-{{n-2}\choose{i-1}}\bigg)\int_{E}\rho^*c_1(\nn)\zeta^{n-2}  \\
 &\ \ \ + \dps(-1)^{n-1}\ell^{n-i-1}\bigg(1-{{n-1}\choose{i-1}}\bigg)\int_{E}\rho^*c_1(C)\zeta^{n-2}.
\end{align*}
Finally,
$$
\dps\int_{E}c_1(\ll_{\widetilde{\fol}}^*)^{n-1}=(-1)^{n-1}\ell^{n-1}\int_{E}\zeta^{n-1}+(-1)^n\ell^{n-2}{{n-1}\choose{1}}\int_{E}\pi^*c_1(\ll_{\fol}^*)\zeta^{n-2}.$$
Using (\ref{forEE}), it follows that
\begin{align*}
\sum_{i=1}^t \mu(\widetilde{\F}|_{E},\widetilde{q}_i) & = \dps (n+1)d\sum_{i=1}^{n-1}\ell^{n-i-1} -\dps((n+1)d-2+2g)\sum_{i=0}^{n-1}\ell^{n-i-1}{{n-1}\choose{i}} \\
&\ \ \ \ \ \ \ \ \ +\dps(k-1)d\sum_{i=0}^{n-1}\ell^{n-i-2}{{n-i-1}\choose{1}}{{n-1}\choose{i}}\\
&\ \ \ \ \ \ \ \ \ \ \ +  ((n+1)d-2+2g)\sum_{i=1}^{n-1}\ell^{n-i-1}\bigg({{n-2}\choose{i-1}}-1\bigg) \\
&\ \ \ \ \ \ \ \ \ \ \ \ \ +   \dps (2-2g)\sum_{i=1}^{n-1}\ell^{n-i-1}\bigg({{n-1}\choose{i-1}}-1\bigg)\\
&=-\dps((n+1)d-2+2g)\sum_{i=0}^{n-1}\ell^{n-i-1}{{n-1}\choose{i}}\\
&\ \ \ \ \ \ \ \ \ +\dps(k-1)d\sum_{i=0}^{n-1}\ell^{n-i-2}{{n-i-1}\choose{1}}{{n-1}\choose{i}}\\
&\ \ \ \ \ \ \ \ \ \ \ +  ((n+1)d-2+2g)\sum_{i=1}^{n-1}\ell^{n-i-1}{{n-2}\choose{i-1}} \\
&\ \ \ \ \ \ \ \ \ \ \ \ \ +   \dps (2-2g)\sum_{i=1}^{n-1}\ell^{n-i-1}{{n-1}\choose{i-1}}\\
&=-\dps((n+1)d-2+2g)(\ell+1)^{n-1}+\dps(k-1)d(n-1)(\ell+1)^{n-2}\\
&\ \ \ \ \ \ \ \ \ +  ((n+1)d-2+2g)(\ell+1)^{n-2}+ \dps (2-2g)\sum_{i=0}^{n-2}(\ell+1)^i
\end{align*}
and it is straight forward obtaining the formula stated in the theorem.
\end{proof}

\begin{example}
\label{exesp}
\emph{Let $\fol$ be a holomorphic foliation by curves of degree $k\geq 2$ on
$\P^n$, induced on the affine open set $U_0 = \{[x_0,\ldots,x_n]\in \P^n\,|\,x_0\ne 0 \}$ by the vector field
$$
\ddd_{\fol} = \dps\sum_{i=1}^{n-1}\bigg(\sum_{|a|=k}c_{i,a}z_1^{a_1}\cdots z_{n-1}^{a_{n-1}}\bigg)\frac{\partial}{\partial
z_i}+ \bigg(\sum_{|a|=k-1}z_1^{a_1}\cdots
z_{n-1}^{a_{n-1}}h_a(z)\bigg)\frac{\partial}{\partial z_n}
$$
where $z_i=x_{i}/x_{0}$, $a=(a_1,\ldots,a_{n-1})$ is a multi-index with $|a|=\sum_{i=1}^{n-1}a_i$, $a_i \ge 0$, $c_{i,a}$ are
constants and $h_a(z)=c_{0,a}'+c_{1,a}'z_1+\ldots+c_{n,a}'z_n$
a linear function. We also consider the $f_i(z)=\sum_{|a|=k}c_{i,a}z^a$ linearly independent over
$\C$.}

\emph{Let $C$ be the curve defined by $x_i = 0$ for
$i=1,\ldots,n-1$. It is a curve of singularities of $\fol$ and we
blowup $\P^n$ along it. In the open set $\widetilde{U}_1$
with coordinates $u \in \C^n$, we have the relations
$$ \sigma(u) = (u_1,u_1u_2,\ldots,u_1u_{n-1},u_n) = z \in \C^n.
$$
Therefore, $\pi^*\fol$ is generated by the vector field
\begin{align*}
\ddd_{\pi^*\fol} & =  \dps u_1^{k}\sum_{|a|=k}c_{1,a}u_2^{a_2}\cdots
u_{n-1}^{a_{n-1}}\frac{\partial}{\partial u_1} +
\sum_{i=2}^{n-1}u_1^{k-1}g_{i,a}(u)\frac{\partial}{\partial
u_i} \\
 &\ \ \ \ +   u_1^{k-1}\sum_{|a|=k-1}h_a(\sigma(u))u_2^{a_2}\cdots
u_{n-1}^{a_{n-1}}\frac{\partial}{\partial u_n},
\end{align*}
where
$$
g_{i,a}(u)=\dps\sum_{|a|=k}c_{i,a}u_2^{a_2}\cdots u_{n-1}^{a_{n-1}}-u_i\sum_{|a|=k}c_{1,a}u_2^{a_2}\cdots u_{n-1}^{a_{n-1}}.
$$
Since $m_{C}(f_i)=m_{C}(f_n)+1 =k$,
for $i=1,\ldots,n-1,$ we have that the multiplicity $\ell :=
m_C(\fol)=\mbox{tang}(\pi^*\fol,E) = k-1$. In this way, the
foliation $\widetilde{\fol}$ induced via $\pi$ is
given in $\widetilde{U}_1$ by the vector field
\begin{align*}
\ddd_{\widetilde{\fol}} & =   u_1\dps\sum_{|a|=k}c_{1,a}u_2^{a_2}\cdots u_{n-1}^{a_{n-1}}\frac{\partial}{\partial u_1}+
\sum_{i=2}^{n-1}g_{i,a}(u)\frac{\partial}{\partial u_i}\\
 &\ \ \ \ \ +\sum_{|a|=k-1}h_a(\sigma(z))u_2^{a_2}\cdots u_{n-1}^{a_{n-1}}\frac{\partial}{\partial
u_n}.
\end{align*}
It is easily seeing that on the affine open set, $u_n \in \C$,
the foliation $\widetilde{\fol}$, when restricted to the
exceptional divisor $E$, which is given by $u_1=0$, defines a
holomorphic foliation on $\P^{n-2}$ of degree $k$ and with infinite
hyperplane noninvariant. Consequently, there are
$\sum_{i=0}^{n-2}k^i$ isolated singularities on $E$ because for each
$(u_2,\ldots,u_{n-1})$ vanishing the $n-2$
first terms of $\ddd_{\widetilde{\fol}|E}$ there is a unique
$u_n$ vanishing the last term of
$\ddd_{\widetilde{\fol}}$, namely,
$\sum_{|a|=k-1}c_{0,a}'u_2^{a_2}\ldots u_{n-1}^{a_{n-1}}+u_n\sum_{|a|=k-1}c_{n,a}'u_2^{a_2}\ldots u_{n-1}^{a_{n-1}}$.
Furthermore, at the fiber $\pi^{-1}[0:0:\cdots:0:1]$ the foliation
$\widetilde{\fol}$ has $\sum_{i=0}^{n-2}k^i$ additional
singularities. Therefore, $\widetilde{\fol}$ when restricted to $E$ has $2\sum_{i=0}^{n-2}k^i$ singularities, which agrees with the number obtained by Theorem \ref{teoA} taking $\ell=k-1$, $g=0$ and $d=1$.}
\end{example}

\begin{theorem}
\label{teoB}
Let $\fol$ has degree $k$, and  multiplicity $\ell$ at $C$; let $C$ has genus $g$ and degree $d$; and let ${\rm Sing}(\widetilde{\fol})=\{\widetilde{p}_1,\ldots,\widetilde{p}_r\}$. Then
\begin{align*}
\sum_{i=1}^r \mu(\widetilde{\F},\widetilde{p}_i) &=1+k+k^2+\ldots+k^n+(2-2g)\bigg(1+(\ell+1)+\ldots+(\ell+1)^{n-3}\bigg)\\
&\ \ \ \ +(\ell+1)^{n-2}\bigg((n+1)d(\ell^2-\ell)-(2-2g)\ell^2-(k-1)d(n\ell-n+2)\bigg)
\end{align*}
\end{theorem}

\begin{proof}
By Baum-Bott's formula, we have that
$$
\sum_{i=1}^r \mu(\widetilde{\F},\widetilde{p}_i) =
\int_{\widetilde{\P}^n}c_n(\tt_{\widetilde{\P}^n} \otimes
\ll_{\widetilde{\fol}}^* )$$
with
$$
c_n(\tt_{\widetilde{\P}^n} \otimes \ll_{\widetilde{\fol}}^* ) = \dps \sum_{i=0}^{n}c_i(\widetilde{\P}^n)\cdot c_1(\ll_{\widetilde{\fol}}^*)^{n-i}.
$$
If $i \ge 2$, the factor $c_i(\pnt)$ is expressed by (\ref{fcc}). And from (\ref{prop03}) we get
$$
c_1(\ll_{\widetilde{\fol}}^*)^{n-i}=\dps\sum_{l=0}^{n-i}{{n-i}\choose{l}}\pi^*c_1(\ll_{{\fol}}^*)^{l}(-\ell E)^{n-i-l}.
$$
Passing from $\pnt$ to $E$ and then to $C$, one sees that
$$
\int_{\pnt}\pi^*c_{i}(\pn)\pi^*c_1(\ll_{\fol}^*)^l E^{n-i-l}=0\ \ \ {\rm for}\ l\neq n-i
$$
$$
\int_{\pnt}\rho^*c_{j}(\nn)\pi^*c_1(\ll_{\fol}^*)^l E^{n-j-l}=0\ \ \ {\rm for}\ j\geq 2\ {\rm or}\ l\geq 2
$$
$$
\int_{\pnt}\rho^*c_{j}(\nn)\rho^*c_1(C)\pi^*c_1(\ll_{\fol}^*)^l E^{n-j-l-1}=0\ \ \ {\rm for}\ j\geq 1\ {\rm or}\ l\geq 1
$$
hence
\begin{align*}
\int_{\widetilde{\P}^n} c_i(\widetilde{\P}^n) c_1(\ll_{\widetilde{\fol}}^*)^{n-i} &=
\int_{\widetilde{\P}^n}\pi^*c_i(\P^n)\pi^*c_1(\ll_{\fol}^*)^{n-i}\\
&+(-1)^{n-1}\ell^{n-i}\bigg({{n-1}\choose{i-1}}-{{n-1}\choose{i}}\bigg)\int_{\widetilde{\P}^n}E^n\\
&+(-1)^{n}\ell^{n-i-1}{{n-i}\choose{1}}\bigg({{n-1}\choose{i-1}}-{{n-1}\choose{i}}\bigg)\int_{\widetilde{\P}^n}\pi^*c_1(\ll_{\fol}^{*})E^{n-1}\\
&+(-1)^{n}\ell^{n-i}\bigg({{n-2}\choose{i-2}}-{{n-2}\choose{i-1}}\bigg)\int_{\widetilde{\P}^n}\rho^*c_1(\nn)E^{n-1}\\
&+(-1)^{n}\ell^{n-i}\bigg({{n-1}\choose{i-2}}-{{n-1}\choose{i-1}}\bigg)\int_{\widetilde{\P}^n}\rho^*c_1(C) E^{n-1}
\end{align*}
which yields
\begin{align*}
\int_{\widetilde{\P}^n} c_i(\widetilde{\P}^n) c_1(\ll_{\widetilde{\fol}}^*)^{n-i} &=
\int_{\pn}c_i(\P^n)c_1(\ll_{\fol}^*)^{n-i}\\
&+(-1)^{n-1}\ell^{n-i}\bigg({{n-1}\choose{i-1}}-{{n-1}\choose{i}}\bigg)\int_{E}\zeta^{n-1}\\
&+(-1)^{n}\ell^{n-i-1}{{n-i}\choose{1}}\bigg({{n-1}\choose{i-1}}-{{n-1}\choose{i}}\bigg)\int_{E}\pi^*c_1(\ll_{\fol}^{*})\zeta^{n-2}\\
&+(-1)^{n}\ell^{n-i}\bigg({{n-2}\choose{i-2}}-{{n-2}\choose{i-1}}\bigg)\int_{E}\rho^*c_1(\nn)\zeta^{n-2}\\
&+(-1)^{n}\ell^{n-i}\bigg({{n-1}\choose{i-2}}-{{n-1}\choose{i-1}}\bigg)\int_{E}\rho^*c_1(C) \zeta^{n-2}
\end{align*}
and thus
\begin{align*}
\int_{\widetilde{\P}^n} c_i(\widetilde{\P}^n) c_1(\ll_{\widetilde{\fol}}^*)^{n-i} &=
(k-1)^{n-i}{{n+1}\choose{i}}\\
&-\ell^{n-i}((n+1)d-2+2g)\bigg({{n-1}\choose{i-1}}-{{n-1}\choose{i}}\bigg)\\
&+\ell^{n-i-1}(k-1)d{{n-i}\choose{1}}\bigg({{n-1}\choose{i-1}}-{{n-1}\choose{i}}\bigg)\\
&+\ell^{n-i}((n+1)d-2+2g)\bigg({{n-2}\choose{i-2}}-{{n-2}\choose{i-1}}\bigg)\\
&+\ell^{n-i}(2-2g)\bigg({{n-1}\choose{i-2}}-{{n-1}\choose{i-1}}\bigg).
\end{align*}
From (\ref{for20}) and (\ref{prop03}) we have that
\begin{align*}
\int_{\widetilde{\P}^n}c_1(\widetilde{\P}^n)c_1(\ll_{\widetilde{\fol}}^*)^{n-1} &=
\int_{\widetilde{\P}^n}\pi^*c_1(\P^n)\pi^*c_1(\ll_{\fol}^*)^{n-1}+(-\ell)^{n-1}\int_{\widetilde{\P}^n}\pi^*c_1({\P}^n)E^{n-1}\\
&\ \ \ \ \ \ \ \ +(-1)^{n-1}\ell^{n}(n-1)(n-2)\int_{\widetilde{\P}^n}\pi^*c_1(\ll_{\fol}^*)E^{n-1}\\
&\ \ \ \ \ \ \ \ \ \ +(-1)^{n}\ell^{n-1}(n-2)\int_{\widetilde{\P}^n}E^n\\
&=(n+1)(k-1)^{n-1}+\ell^{n-1}(n-2)((n+1)d-2+2g)\\
&\ \ \ \ \ \ \ \ -\ell^{n-2}(n-1)(n-2)(k-1)d-\ell^{n-1}(n+1)d.
\end{align*}
Finally, from (\ref{prop03}) we get
$$
\int_{\widetilde{\P}^n}c_1(\ll_{\widetilde{\fol}}^*)^n=(k-1)^n+\ell^n((n+1)d-2+2g)-\ell^{n-1}(k-1)nd.
$$
Summing up the cases $i=0$, $i=1$ and $i\geq 2$, we obtain
\begin{align*}
\sum_{i=1}^r \mu(\widetilde{\F},\widetilde{p}_i)& = \sum_{i=0}^{n}(k-1)^{n-i}{{n+1}\choose{i}}-\ell^{n-1}(n+1)d \\
&\ \ \ \ \ \ \ \ \  -((n+1)d-2+2g)\sum_{i=0}^{n}\ell^{n-i}\bigg({{n-1}\choose{i-1}}-{{n-1}\choose{i}}\bigg)\\
&\ \ \ \ \ \ \ \ \ \ \  + (k-1)d\sum_{i=0}^{n}\ell^{n-i-1}{{n-i}\choose{1}}\bigg({{n-1}\choose{i-1}}-{{n-1}\choose{i}}\biggr)\\
&\ \ \ \ \ \ \ \ \ \ \ \ \ + ((n+1)d-2+2g)\sum_{i=2}^{n}\ell^{n-i}\bigg({{n-2}\choose{i-2}}-{{n-2}\choose{i-1}}\bigg)\\
&\ \ \ \ \ \ \ \ \ \ \ \ \ \ \ + (2-2g)\sum_{i=2}^{n}\ell^{n-i}\bigg({{n-1}\choose{i-2}}-{{n-1}\choose{i-1}}\bigg)\\
&=\sum_{i=0}^{n}k^i-\ell^{n-1}(n+1)d+(2-2g)\bigg(\sum_{i=0}^{n-2}(\ell+1)^i-(\ell+1)^{n-1}+\ell^{n-1}\bigg)\\
&-((n+1)d-2+2g)(\ell+1)^{n-2}(\ell^2-1)-(k-1)d(\ell+1)^{n-2}(n\ell-n+2)\\
&+((n+1)d-2+2g)((\ell+1)^{n-2}(1-\ell)+\ell^{n-1})
\end{align*}
and the desired formula is straight forward.
\end{proof}

As a consequence we get Theorem 1 in the case we are dealing with here.

\begin{corolary}
\label{cort1r}
Let $\fol$ has degree $k$, and multiplicity $\ell$ at $C$; let $C$ has genus $g$ and degree $d$. Then
\begin{align*}
\sum_{i=1}^{s}\mu(\fol,p_i) &=1+k+k^2+\ldots+k^n \\
&\ \ \ \ +(\ell+1)^{n-2}\bigg((2g-2)(\ell^2+\ell+1)+(n+1)d\ell^2-(k-1)d(n\ell+1)\bigg)
\end{align*}
\end{corolary}

\begin{proof}
Just note that
$$
\sum_{i=1}^{s}\mu(\fol,p_i)=\sum_{i=1}^r \mu(\widetilde{\F},\widetilde{p}_i)-\sum_{i=1}^t \mu(\widetilde{\F}|_{E},\widetilde{q}_i)
$$
and recall Theorems \ref{teoA} and \ref{teoB}.
\end{proof}

Let $\fol$ be as in the Example \ref{exesp}. It has no singularities in $U_0$ but the ones in $C \cap U_0$. The hyperplane $H_0 := \P^n\setminus U_0$ is isomorphic to $\P^{n-1}$ as well as invariant by $\fol$. As the degree of $\fol|_{H_0}$ remains $k$, the number of isolated singularities, counted with multiplicities, of $\fol$ on $H_0$ is $\sum_{i=0}^{n-1}k^i$. Given that the singularity $q=[0:0:\cdots 0:1] \in C$ has Milnor number
$\mu(\fol|_{H_0},q) = k^{n-1}$, it follows that $\fol$ has $\sum_{i=0}^{n-2}k^i$ isolated singularities in $\P^n$, counted
with multiplicities, which agrees with the number obtained by Corolarry \ref{cort1r} taking $\ell=k-1$, $g=0$ and $d=1$.

\section{Special Foliations along Complete Intersections}

The aim of this section is proving Theorem 1 on its full generality. In order to do so we need a result on foliations admiting a complete intersection curve in its singular locus.

\begin{lema}
\label{lema4}
Let $\fol$ be a holomorphic foliation by curves on $\P^n$, $n\ge 3$, with
$$
\label{for13} \mbox{Sing}(\fol) = C \cup
\{p_1,\ldots,p_s\},
$$
where the union is disjoint, $C$ is an irreducible singular projective curve, the $p_i$ are isolated closed points, and $\fol$ is special along $C$. Then there exists a one-parameter family of holomorphic foliations by curves on $\pn$, given by $\{\fol_t\}_{t\in D}$ where $D=\{t\in\C\,|\,|t|<\epsilon\}$ such that
\begin{enumerate}
\item[(i)] $\fol_0=\fol$;
\item[(ii)] $\deg(\fol_t)=\deg(\fol)$;
\item[(iii)] $\sing(\fol_t)=C_t\cup\{p_1^t,\ldots,p_{s_t}^t\}$, where $C_t$ is a regular irreducible projective curve with $\deg(C_t)=\deg(C)$, and the $p_i^t$ are closed points;
\item[(iv)] $\fol_t$ is special along $C_t$ and $m_{C_t}(\fol_t)=m_{C}(\fol)$;
\item[(v)] $\dps\sum_{i=1}^{s_t}\mu(\fol_t,p_i^t)=\sum_{i=1}^{s}\mu(\fol,p_i)$
\end{enumerate}
\end{lema}

\begin{proof}
Assume $C$ is given, in an affine standard chart of $\pn$, by the zeros of the polynomials $f_1,\ldots,f_{n-1}$. Take polynomials $h_1,\ldots,h_{n-1}$ and consider the holomorphic function
\begin{gather*}
\begin{matrix}
F_t\  : &  \C^n\  & \longrightarrow & \C^{n}   \\
      & z=(z_1,\ldots,z_n)  & \longmapsto     & \bigg(f_1(z)+th_1(z)\, ,\,\ldots\, ,\, f_{n-1}(z)+th_{n-1}(z)\, ,\, z_n\bigg).
\end{matrix}
\end{gather*}
For each $t\in\C$ and any $z\in\C^n$, set $M_t(z)$ to be the first $(n-1)\times(n-1)$-minor of the jacobian matrix $D_z F_t$. Define $U_t:=\C^n\setminus\{\det M_t=0\}$ and let $C_t$ be the projective closure of the common zeros locus of the $f_i+th_i$. Note that $F_t|_{U_t}$ is a local biholomorphism onto an open set $V_t\subset\C^n$ and the image of $C_t\cap U_t$ by $F_t$ is the $w_n$-axis restricted to $V_t$ so one may fix coordinates $w=F_t(z)$. In particular, we may describe the pushforward $(F_0)_*\fol$ in $V_0$, as in (\ref{for40}), by the vector field
$$
\ddd_{(F_0)_*\fol}=P_1\,\frac{\partial}{\partial w_1}+\ldots +P_n\,\frac{\partial}{\partial w_n}
$$
where
\begin{equation}
\label{equgil}
P_i(w)=\sum_{|a|=m_i}w_1^{a_1}\cdots w_{n-1}^{a_{n-1}}P_{i,a}(w)
\end{equation}
with at least one $P_{i,a}(z)$ not vanishing in the $w_n$-axis.

For each $t\in\C$, we define $\fol_t$ by the vector field
$$
\ddd_{\fol_t}=Q_1^t\,\frac{\partial}{\partial z_1}+\ldots +Q_n^t\,\frac{\partial}{\partial z_n}
$$
where the $Q_i^t$ are obtained by the system
\begin{equation}
\label{equmat}
\begin{pmatrix}
P_1\circ F_t \cr \vdots\cr P_{n}\circ F_t \end{pmatrix}
\,=\,
DF_t\,\cdot\,
\begin{pmatrix} Q_1^t \cr
\vdots\cr
Q_{n}^t \end{pmatrix}
\end{equation}
and $DF_t$ is the Jacobian matrix
$$
\begin{pmatrix}
\partial(f_1+th_1)/\partial z_1 & \ldots & \partial(f_1+th_1)/\partial z_{n-1} & \partial(f_1+th_1)/\partial z_{n} \cr
\vdots & \ddots & \vdots & \vdots \cr
\partial(f_{n-1}+th_{n-1})/\partial z_1 & \ldots & \partial (f_{n-1}+th_{n-1})/\partial z_{n-1} & \partial(f_{n-1}+th_{n-1})/\partial z_{n} \cr
0 & \ldots & 0 & 1
\end{pmatrix}.
$$
Solving the system by Cramer's rule, we have
$$
Q_i^t = \frac{\det A_i^t}{\det M_t}
$$
where one gets $A_i^t$ replacing the $i$th column of $DF_t$ by the
column vector at the left hand side of the equality (\ref{equmat}). In
particular,
$$
Q_{n}^t = \frac{P_n\circ F_t \cdot
\det M_t}{\det M_t}=P_n\circ F_t.
$$
Therefore, normalizing by the factor $\det M_t$, one may describe $\fol$ in $U_t$ by
\begin{equation}
\label{for80}
\ddd_{\fol_t} =\det A_1^t\,\frac{\partial}{\partial z_1}+\ldots+\det A_{n-1}^t\,\frac{\partial}{\partial z_{n-1}}+P_n\circ F_t\,\frac{\partial}{\partial z_n}.
\end{equation}
As the components of $\ddd_{\fol_t}$ are polynomials, using Hartogs Extension
Theorem, we can consider $\ddd_{\fol_t}$ defined in $\C^n$.

It is immediate that $\fol_0=\fol$. Besides, the $\fol_t$ were built targeting property (iv) above. In fact, the pushforwards $(F_t)_*\fol_t$ agree with $(F_0)_*\fol$ no matter is $t\in\C$, hence, by (\ref{equgil}) and Lemma \ref{lemlll}.(i),(iii), it follows that $m_{C_t}(\fol_t)=m_C(\fol)$ and $\fol_t$ special along $C_t$ for every $t\in\C$, because $\F$ is so. By construction, $F_t$ is a local biholomorphism, so $\sing(\fol_t)$ must be the disjoint union of $C_t$ and points, with $C_t$ irreducible since its image by $F_t$ is a line. Assuming $\mbox{deg}(h_i)\leq\deg(f_i)$ for $1\leq i\leq n-1$, one assures that $\deg(C_t)=\deg(C)$, and also, by (\ref{for80}), that $\deg(\fol_t)=\deg(\fol)$.

Now set $\C^{n+1}=\{(z,t)\in\C^{n}\times\C\}$. Note that the family $S:=\cup_{t\in\C}\,(C_t\cap\C^{n})$ is an algebraic surface in $\C^{n+1}$. On the other hand, singularity imposes $n$ conditions by the vanishing of the $(n-1)\times(n-1)$-minors of the jacobian matrix $DF_t(z)$ and generically determines an algebraic curve in $\C^{n+1}$. If these two varieties happen to meet, which is the case since $C$ is singular, they do generically at isolated closed points, so one may adjust the $h_i$ and find $\epsilon >0$ sufficiently small such that $C_t\cap\C^n$ is regular for $0<|t|<\epsilon$. The whole family surface $\overline{S}:=\cup_{t\in\C}\,C_t$ intersects $H:=(\pn\setminus\C^{n})\times\C$ at a curve. Change coordinates (a priori) and assume $C$ has no singular points at $H$. Since singularity is a closed algebraic condition, either finitely many points of $\overline{S}\cap H$ are singular points of their curves, or all of them are so. But this contradicts our assumption, so one may take $\epsilon$ smaller if necessary to get the whole $C_t$ regular if $0<t<\epsilon$. And one may shrink $\epsilon$ even more in order to have
$$
\sum_{i=1}^{s_t}\mu(\fol_t,p_i^t)=\sum_{i=1}^{s}\mu(\fol,p_i)
$$
as well.
\end{proof}

Now Theorem 1 is straight forward. Suppose $\singf$ has a unique nonzero dimensional component $C$. Use Lemma \ref{lema4} to pick up any $\fol'$ among
the $\{\fol_t\}_{t\in D\setminus 0}$. Apply Corollary \ref{cort1r} to $\fol'$ and get a formula for $\sum_{i=1}^{s'}\mu(\fol',p_i')$. By Lemma \ref{lema4}, this formula holds for $\sum_{i=1}^{s}\mu(\fol,p_i)$ with same $d,\ell,k$ up to the factor $2g-2$. So let $g'$ be the genus of $C'$ and let $\overline{g}$ and $\delta$ be respectively the geometric genus and the cogenus of $C$. Since there is a continuos deformation from $C$ to $C'$, and $C$ is a set theoretic complete intersection, we have
\begin{align*}
2g'-2 &=-\chi(C')=-\chi(C)+\sum_{i=1}^{l}\mu(C,q_i) \\
      &=2\overline{g}-2+2\delta -\sum_{i=1}^{l}(b_i-1) \\
      &=2(\overline{g}+\delta)-2-\sum_{i=1}^{l}(b_i-1) \\
      &=2g-2-\sum_{i=1}^{l}(b_i-1)
\end{align*}
and one gets Theorem 1 for $\fol$ in the case there is just one $C$. If there are many, say, $C_1,\ldots,C_r$, set $Y_0 = \P^n$ and take a sequence of blowups $\pi_i: Y_i \to Y_{i-1}$ centered at $C_{i}$ with exceptional divisor $E_i$. Then it is just a matter of slightly adjusting succesively the proof of Theorem \ref{teoB} just noticing that $E_i\cdot E_j=0$ if $i \ne j$ because the curves $C_i$ and $C_j$ are
disjoint.

\section{The Unicity Problem}

In this section we prove Theorems 2 and 3. For the remainder, $\fol$ is a foliation by curves on $\P^n$, $n\ge 3$, with
\begin{equation}
\label{for13} \mbox{Sing}(\fol) = \dps C \cup
\{p_1,\ldots,p_s\}
\end{equation}
where the union is disjoint, $C$ is an integral and smooth projective curve, the $p_i$ are closed points, $\fol$ is special along $C$, and $\pi:\pnt\to\P^n$ is the blowup of $\pn$ along $C$ with exceptional divisor $E$.

We also recall from the Introduction that ${\rm gd}(C)$ is the \emph{generating degree of $C$}, the least integer $d>0$ such that $\mathcal{I}_C(d)$ is globally generated.

\begin{theorem}
Let $\F'$ be a holomorphic foliation by curves on $\pn$ for which we have $\deg(\F')=\deg(\F)>{\rm gd}(C)$, and such that ${\rm Sing}(\F)\subset{\rm Sing}(\F')$, and also ${\rm Sing}(\widetilde{\F}|_{E})\subset{\rm Sing}(\widetilde{\F}'|_{E})$. Then $\F'=\F$.
\end{theorem}

\begin{proof}
Set $\deg(\F)=\deg(\F')=k$. Then $\F$ and $\F'$ are induced by sections
$$
s_{\F},s_{\F'}\in H^0(\tt_{\pn}\otimes \mathcal{O}_{\pn}(k-1)).
$$
Similarly, $\widetilde{\F}$ and $\widetilde{\F}'$ are induced by
$$
s_{\widetilde{\F}},s_{\widetilde{\F}'}\in H^0(\pi^*\tt_{\pn}\otimes\mathcal{O}_{\pnt}(k-1)\otimes
\mathcal{O}_{\pnt}(- E))
$$
because $\ll_{\widetilde{\F}}^*=\ll_{\widetilde{\F}'}^*=\mathcal{O}_{\pnt}(k-1)\otimes
\mathcal{O}_{\pnt}(-\ell E)$ and $\ell=1$ since $C$ is integral and $\F$ is special along $C$.

We will prove that
\begin{equation}
\label{equsec}
s_{\widetilde{\F}}=\lambda\cdot s_{\widetilde{\F}'}\ \ {\rm iff}\ \ \ {\rm Sing}(\widetilde\F)\subset{\rm Sing}(\widetilde{\F}')
\end{equation}
for some $\lambda \in \mathbb{C}^*$ and $k>{\rm gd}(C)$. If so, we get the statement of the theorem since ${\rm Sing}(\F)\subset{\rm Sing}(\F')$ and ${\rm Sing}(\widetilde{\F}|_{E})\subset{\rm Sing}(\widetilde{\F}'|_{E})$ imply that ${\rm Sing}(\widetilde\F)\subset{\rm Sing}(\widetilde{\F}')$, and, by projection, $s_{\widetilde{\F}}=\lambda\cdot s_{\widetilde{\F}'}$ implies $s_{\F}=\lambda\cdot s_{\F'}$, and the latter yields $\F'=\F$.

In order to get (\ref{equsec}), adjusting the proofs of \cite[Thm 2.2, Cor 3.2]{C-O}, with $\oo_{\pnt}(1)$ playing the role of an ample bundle, it suffices checking that $\pi^*\tt_{\pn}\otimes\mathcal{O}_{\pnt}(- E)$ is simple and that
\begin{equation}
\label{equcnd}
H^p(\pi^*\Omega_{\pn}^q\otimes \pi^*\tt_{\pn} \otimes \pi^*\mathcal{O}_{\pn}((1-q)(k-1))\otimes\mathcal{O}_{\pnt}((q-1)E))=0
\end{equation}
for $2\leq q\leq n$ and $p=q-2,q-1$.

Simplicity immediately follows from projection formula
$$
H^0(\pi^*\Omega_{\pn}^1\otimes \pi^*\tt_{\pn})\simeq H^0(\Omega_{\pn}^1\otimes \tt_{\pn})\simeq\mathbb{C}
$$
while (\ref{equcnd}), within the desired range, deserves more care.

From \cite[Lem. 1.4]{BELa} we know that if $0 \leq t \leq n-2$, then
\begin{equation}
\label{equ333}
H^i(\pi^*F\otimes \mathcal{O}_{\pnt}(tE))\simeq H^i(F)
\end{equation}
for all $i\in\mathbb{N}$ and any locally free sheaf $F$. On the other hand, from \cite{GK},
\begin{equation}
\label{equ999}
H^p(\Omega_{\pn}^q\otimes \tt_{\pn} ((1-q)(k-1)))=0
\end{equation}
for $k\geq 0$ and $p<q$, $2\leq q\leq n$.

Therefore, from (\ref{equ333}) and (\ref{equ999}), we have for $2\leq q\leq n-1$
$$
H^p(\pi^*\Omega_{\pn}^q\otimes \pi^*\tt_{\pn} \otimes \pi^*\mathcal{O}_{\pn}((1-q)(k-1))\otimes\mathcal{O}_{\pnt}((q-1)E))=
$$
$$
=H^p(\Omega_{\pn}^q\otimes \tt_{\pn} ((1-q)(k-1)))=0.
$$

Now we analyze the case  $q=n$, that is, the vanishing of
$$
H^p(\pi^*\Omega_{\pn}^n\otimes \pi^*\tt_{\pn} \otimes \pi^*\mathcal{O}_{\pn}((1-n)(k-1))\otimes\mathcal{O}_{\pnt}((n-1)E))
$$
for $p= n-2, n-1$. Observe that above groups are
$$
H^p(\pi^*\tt_{\pn} \otimes \pi^*\mathcal{O}_{\pn}((1-n)(k-1))\otimes\mathcal{O}_{\pnt}(E)\otimes \omega_{\pnt})
$$
since  the dualizing sheaf on $\pnt$ is $\omega_{\pnt}=\pi^*\Omega_{\pn}^n\otimes \mathcal{O}_{\pnt}((n-2)E)$.  By
Serre's duality
$$
H^{n-i}(\pi^*\tt_{\pn} \otimes
\pi^*\mathcal{O}_{\pn}((1-n)(k-1))\otimes\mathcal{O}_{\pnt}(E)\otimes
\omega_{\pnt})\simeq
$$
$$
\simeq H^{i}(\pi^*\Omega^1_{\pn} \otimes
\pi^*\mathcal{O}_{\pn}((n-1)(k-1))\otimes\mathcal{O}_{\pnt}(-E)).
$$
Then, we must have to prove that
$$
H^{i}(\pi^*\Omega^1_{\pn} \otimes
\pi^*\mathcal{O}_{\pn}((n-1)(k-1))\otimes\mathcal{O}_{\pnt}(-E))=0
$$
for $i=1,2.$

Since $\pi_*\mathcal{O}_{\pnt}(-E)=\mathcal{I}_C$ and
$\mathrm{R}^i\pi_*\mathcal{O}_{\pnt}(-E)=0$, it follows from the
projection formula and Leray's spectral sequence that
$$
H^{i}(\pi^*\Omega^1_{\pn} \otimes
\pi^*\mathcal{O}_{\pn}((n-1)(k-1))\otimes\mathcal{O}_{\pnt}(-E))\simeq
$$
\begin{equation}
\label{equ12f}
 H^{i}(\Omega^1_{\pn} \otimes  \mathcal{I}_C((n-1)(k-1)))
\end{equation}
so we just have to check the vanishing of (\ref{equ12f}) for $i=1,2$. In order to get this, by Mumford's regularity theorem, it suffices to show that  $(n-1)(k-1)\geq m-1$ if $\Omega^1_{\pn} \otimes  \mathcal{I}_C$ is $m$-regular.

From Bott's formulae, $ \Omega^1_{\pn} $ is $2$-regular, while
$\mathcal{I}_C$ is $((n-1){\rm gd}(C) - n + 2)$-regular by \cite{BELa}.
Hence $\Omega^1_{\pn} \otimes  \mathcal{I}_C$ is $((n-1){\rm gd}(C)
- n + 4)$-regular owing to \cite[Prp. 1.8.9]{La}.

But, by hypothesis,
$$
(n-1)(k-1)\geq (n-1){\rm gd}(C)\geq (n-1){\rm gd}(C) - n +3
$$
for $n\geq 3$ and we are done.
\end{proof}

In the case of three dimensional projective space, we can also get the following.

\begin{theorem}
Let $\F'$ be a foliation on $\mathbb{P}^3$ with $\deg(\F')=\deg(\F)$, such that ${\rm Sing}(\F)\subset{\rm Sing}(\F')$. If $C$ is also nondegenerated and a set theoretic complete intersection in $\mathbb{P}^3$, then $\F'=\F$.
\end{theorem}

\begin{proof}
Set $\deg(\F)=\deg(\F')=k$. According to the prior result, we just have to prove that, in $\mathbb{P}^3$,
 we always have $k>{\rm gd}(C)$; and ${\rm Sing}(\widetilde{\F}|_{E})\subset{\rm Sing}(\widetilde{\F}'|_{E})$ always holds as well.

For the first, under the hypothesis on $C$, assume it is given by
the intersection of surfaces of degree $d_1$ and $d_2$, with
$d_2\geq d_1\geq 2$. It follows from \cite[Lem. 3.6]{GNC} $\deg(\F)
\geq  (m_C(\F) + 1)d_2 + d_1 - 2$. Since $\F$ is special along $C$,
which is integral, we have
$$
k\geq 2d_2 + d_1 - 2\geq 2d_2> d_2+1\geq {\rm gd}(C)+1
$$
because, by definition, $d_2 \geq {\rm gd}(C)$.

For the second, just recall that $E$ is naturally identified with
the projectivization $\mathbb{P}(\nn_{C/\mathbb{P}^3})$ of the
normal bundle $\nn_{C/\mathbb{P}^3}$. Thus, since the vector fields
inducing the foliations $\F$ and  $\F'$ vanish identically along
$C$, the lifts $\widetilde{\F}|_E$ and  $\widetilde{\F}'|_E$ must be
tangent to the the fibers of $\nn_{C/\mathbb{P}^3}$, and hence
coincide.
\end{proof}

Now we show that the assumption on the curve to be nondegenerated is necessary. In fact, we build a family of foliations by curves on $\mathbb{P}^{3}$ with same singular locus consisting of a degenerated smooth curve and isolated points, and all of them special along this curve. 

\par Let $\F_t$ be a holomorphic foliation by curves on $\mathbb{P}^{3}$, with $t \in \C$, induced on the affine open set $U_0 = \{ [x_0,x_1,x_2,x_3]\in \P^3\,|\, x_0 \ne 0\}$ by the vector field

\begin{align*}
\mathcal{D}_{\F_t}&= \bigg(a_0z_1^2+a_1z_1z_2+a_2z_2^2\bigg)\,\frac{\partial}{\partial z_1} + \bigg(b_0z_1^2+b_1z_1z_2+b_2z_2^2\bigg)\,\frac{\partial}{\partial z_2} \\
&\ \ \  +\bigg(z_1\bigg(\alpha_0+\alpha_1z_1+(\alpha_2-t)z_2+\alpha_3z_3\bigg)+z_2\bigg(\beta_0+tz_1+\beta_2z_2+\beta_3z_3\bigg)\bigg)
 \frac{\partial}{\partial z_3} 
\end{align*}
with $z_i=x_{i}/x_0$.

Assume the polynomials $\sum_{j=0}^2a_j\lambda^j$ and $\sum_{j=0}^2b_j\lambda^j$ have no common roots and let $\lambda_i$, for $i=1,2,3$, be the roots of 
$$
\sum_{j=0}^{2}b_j\lambda^j- \lambda\sum_{j=0}^{2}a_j\lambda^j
$$
One can check that $\mathrm{Sing}(\fol_t)$ is the union of the curve $C:=\{x_1=x_2=0\}$ and the points
$$
[0:u_i:\lambda_i u_i: 1]\in \P^3
$$
where
$$
u_i = \frac{a_0 + a_1\lambda_i+a_2\lambda_i^2-\alpha_3 -\beta_3\lambda_i}{\alpha_1+\alpha_2\lambda_i+\beta_2\lambda_i^2}
$$
so $\mathrm{Sing}(\fol_t)$ does not depend on $t$. Note also that the foliation $\widetilde{\F}_t$ induced by
$\F_t$ via the blowup of $\P^3$ along $C$, has the same singular
locus.

\

\




\end{document}